\documentclass[11pt]{article}

\usepackage[dvips]{graphicx}
\usepackage{amssymb}
\usepackage{amsmath}
\usepackage{amsthm}
\usepackage{latexsym}
\usepackage[english]{babel}
\usepackage{appendix}
\usepackage{mathtools}

\usepackage[usenames,dvipsnames]{color}

\newtheorem{theorem}{Theorem}[section]
\newtheorem{definition}{Definition}[section]
\newtheorem{lemma}{Lemma}[section]
\newtheorem{proposition}{Proposition}[section]
\newtheorem{remark}{Remark}[section]

\newtheorem{corollary}{Corollary}[section]

\newcommand{\Sph}{\mathbb{S}}
\def \C {\mathbb{C}}
\def \R {\mathbb{R}}
\def \H {\mathbb{H}}
\def \N {\mathbb{N}}
\def \S {\mathcal{S}}
\def \Cay {\mathcal{C}}

\def \n {\nabla}

\def \d {\textnormal{d}}
\def \Lap {\triangle}

\newcommand{\ol}[1]{\overline{#1}}
\newcommand{\thetaH}[1]{\theta^{\H}_{#1}}
\newcommand{\thetaS}[1]{\theta^{\mathbb{S}}_{#1}}

\begin{document}

\title{Singular CR structures of constant Webster curvature and applications}

\author{Chiara Guidi$^{(1)}$ \& Ali Maalaoui$^{(2)}$ \& Vittorio Martino$^{(3)}$}
\addtocounter{footnote}{1}
\footnotetext{Dipartimento di Matematica, Universit\`a di Bologna, piazza di Porta S.Donato 5, 40126 Bologna, Italy. E-mail address:
{\tt{chiara.guidi12@unibo.it}}}
\addtocounter{footnote}{1}
\footnotetext{Department of mathematics and natural sciences, American University of Ras Al Khaimah, PO Box 10021, Ras Al Khaimah, UAE. E-mail address:
{\tt{ali.maalaoui@aurak.ae}}}
\addtocounter{footnote}{1}
\footnotetext{Dipartimento di Matematica, Universit\`a di Bologna, piazza di Porta S.Donato 5, 40126 Bologna, Italy. E-mail address:
{\tt{vittorio.martino3@unibo.it}}}

\date{}
\maketitle

\vspace{5mm}

{\noindent\bf Abstract} {\small We consider the sphere $\Sph^{2n+1}$ equipped with its standard CR structure. In this paper we construct explicit contact forms on $\Sph^{2n+1}\setminus \Sph^{2k+1}$, which are conformal to the standard one and whose related Webster metrics have constant Webster curvature; in particular the curvature is positive if $2k< n-2$. As main applications, we provide two perturbative results. In the first one we prove the existence of infinitely many contact structures on $\Sph^{2n+1}\setminus \tau(\Sph^{1})$ conformal to the standard one and having constant Webster curvature, where $\tau(\Sph^{1})$ is a small perturbation of $\Sph^1$. In the second application, we show that there exist infinitely many bifurcating branches of periodic solutions to the CR Yamabe problem on $\Sph^{2n+1}\setminus \Sph^{1}$ having constant Webster curvature. }

\vspace{5mm}


\vspace{4mm}


\vspace{4mm}


\section{Introduction  and statement of the results}

\noindent
Let $n \geq 1$, we consider the sphere $\Sph^{2n+1}$ equipped with its standard (flat) CR contact structure $\thetaS{n}$. The related Webster metric $g_{\thetaS{n}}$ has constant Webster scalar curvature $\S_{\thetaS{n}}=4n^2+4n$. The existence of conformal contact forms on the sphere having constant curvature is the standard CR Yamabe problem, which has been addressed by Jerison and Lee and many other authors (see \cite{jerison-lee 1987, jerison-lee 1988}). \\
As in the Riemannian case, one is then interested in the existence of CR contact structures on non-compact manifolds, which carry a (complete) Webster metric having constant Webster curvature. In the Riemannian case, this question has been deeply studied. In fact, one finds two directions in the literature. The first one addresses the case of negative constant scalar curvature, see for instance \cite{Loewner-Nirenberg 1974, Aviles-McOwen 1985, Aviles-McOwen 1988, Aviles-McOwen 2 1988}.\\
The second case addresses metrics of positive constant scalar curvature, starting by the pioneering works of Schoen and Yau \cite{Schoen-Yau 1988} and Schoen \cite{Schoen 1988}: in particular, when considering a subset $\Lambda$ on the standard sphere $\Sph^n$, it is proved that if $\Sph^{n}\setminus \Lambda$ carries a complete metric with positive scalar curvature then a bound on the dimension of $\Lambda$ holds. More precisely, $2 \textit{dim}(\Lambda) \leq n-2$; moreover explicit examples are given of complete conformally flat metrics with constant positive scalar curvature on special domains $\Lambda$.\\
These results have been widely used and generalized in various directions: see for instance \cite{MP}, \cite{MP2},\cite{MS}, \cite{BPS}, \cite{LPZ}, and the references therein. In fact, one can prove the existence of complete conformally flat metrics with constant positive scalar curvature on $\Sph^{n}\setminus \Lambda$ where $\Lambda$ is a perturbation of some special sets, namely the equatorial spheres $\Sph^{k}\subseteq \Sph^{n}$ (\cite{MS}); also, by means of the theory of bifurcation, one can show the existence of periodic solutions to the standard Yamabe problem on $\Sph^{n}\setminus \Sph^{1}$ (\cite{BPS}): in these kind of results, the starting point is the knowledge of explicit complete conformally flat metrics with constant positive scalar curvature on the special manifolds $\Sph^{n}\setminus \Sph^{k}$.\\
In this paper we will show the existence of explicit complete conformally flat CR structures on $\Sph^{2n+1}\setminus \Sph^{2k+1}$, whose related metrics have constant Webster curvature; in particular the curvature is positive if $2k < n-2$.\\
Our construction mimics the one in the Riemannian case. In fact, we first project stereographically (by means of the Cayley transform, which is a conformal transformation) the standard sphere $\Sph^{2n+1}$ to the Heisenberg group $\H^n$ endowed with its standard contact form $\thetaH{n}$, in such a way that the equatorial sphere $\Sph^{2k+1}$ is mapped into the subgroup $\H^k$. Then in the complementary set we use polar coordinates, so that (with some abuse of notation, which will be explained in details in the following sections) we have the product manifold  $\H^n=\H^k\times\R^+\times\Sph^{2N+1}$, endowed with the contact form $\thetaH{n}=2r^2\thetaS{N}+\thetaH{k}$; here $n=k+N+1$, and $r$ is the variable of the polar coordinates in $\R^+$. At this point, we have the following:

\begin{theorem}\label{thm: explicit singular contact form}
Let us define the following contact form $\theta_{k,N}:=\thetaS{N}+\frac{1}{2r^2}\thetaH{k}$ on $\H^{n}\setminus \H^{k} \simeq \Sph^{2n+1}\setminus \Sph^{2k+1}$. It holds that $\theta_{k,N}$ is conformal to the standard CR contact structure $\thetaS{n}$ of $\Sph^{2n+1}$. Moreover, the related Webster metric is complete and it has constant Webster scalar curvature $\S_{\theta_{k,N}}=4(n+1)(n-2k-2)$. In particular we have that $\S_{\theta_{k,N}}$ is  positive for $2k < n-2$.
\end{theorem}

\noindent
To the best of our knowledge, this is the first result in this direction. Now some remarks are in order. First we notice that our construction works fine for the odd dimensional equatorial spheres $\Sph^{2k+1}$; we are not able to handle the even dimensional case with this strategy. Another interesting feature, which seems to be different from the Riemannian case, is the following. In the classic case one can see the product $\R^n=\R^k\times\R^+\times\Sph^{N}=\mathcal{H}^{k+1}\times\Sph^{N}$, where $\mathcal{H}^{k+1}$ is the standard hyperbolic space which in turn can be identified with the unit ball in $\R^{k+1}$ equipped with the Poincar\'{e} metric, having negative constant sectional curvatures. For the CR case, in literature there exists a standard hyperbolic Heisenberg group $\H^k\times\R^+$, which can be seen as the upper half Siegel domain in $\C^{k+1}$ or equivalently as the unit ball in $\C^{k+1}$ equipped with the K\"{a}hler Bergman metric, having negative constant holomorphic curvatures (see for instance \cite{FGMT}). Now, if one tries to write the product $\H^k\times\R^+\times\Sph^{2N+1}$ endowed with the contact form $\theta_{k,N}$ as a product of a sort of hyperbolic Heisenberg group times the sphere $\Sph^{2N+1}$, this gives rise to a different model, that is: the contact structure $J$ associated to $\theta_{k,N}$ \emph{mixes} vector fields from the Heisenberg group $\H^k$ and the sphere $\Sph^{2N+1}$; this will be clear from the explicit construction in the following sections.\\
With these explicit contact structures in hands, as applications we will prove two perturbative results. The first one is analogous to a result proved by Mazzeo and Smale in \cite{MS}, which gives existence of CR contact structures having constant Webster curvature by means of a small perturbation of the singular set. More precisely, we have:

\begin{theorem}\label{thm: perturbation}
Let $\tau : \Sph^{2n+1}\to \Sph^{2n+1}$ be a smooth diffeomorphism which is close to the identity. Then there exists an infinite family of contact structures on $\Sph^{2n+1}\setminus \tau(\Sph^{1})$ conformal to the standard one in $\Sph^{2n+1}$, having complete Webster metric with constant Webster scalar curvature equals to $\S_{\theta_{0,n-1}}$.
\end{theorem}

\noindent
The second application is about the existence of periodic solutions to the CR Yamabe equations, as in \cite{BPS}, which is obtained by using the theory of bifurcation.

\begin{theorem}\label{thm: bifurcation}
Assume that  $n\geq 3$. There exist infinitely many branches of periodic solutions to the CR Yamabe problem on $\Sph^{2n+1}\setminus \Sph^{1}$ having constant Webster curvature, arbitrary close to $\S_{\theta_{0,n-1}}$.
\end{theorem}

\section{Definitions and notation}

\noindent
We recall here some well known facts for further references and in order to fix our notations.\\
Let $(M^{2n+1},\theta)$, $M^{2n+1}\subset \C^{n+1}$ be $2n+1$ dimensional contact manifold with contact form $\theta$ and Reeb vector field $T$ (i.e. the unique vector field satisfying $\theta(T)=1$ and $\d\theta(T,\cdot)=0$). We set $g_{\theta}$, the Webster metric, which is a Riemmanian metric associated to $\theta$, and a $(1,1)$-tensor $\phi$ satisfying:
\begin{equation}\label{eq: metric associated to theta}
g_{\theta}(T,X)=\theta(X),\quad g_{\theta}(X,Y)=-\frac{1}{2}d\theta(X,\phi Y),\quad \phi\phi X=-X+\theta(X)T.
\end{equation}
We define $J=\phi|_{\text{ker}(\theta)}$ (see \cite{Tan89}). If $g_{\theta}$ is a Riemmanian metric associated to $\theta$, then $(M,\theta, g_{\theta}, \phi)$ is called a contact Riemmanian manifold. We denote by $\Lap_{g_{\theta}}$ the metric Laplacian and we consider the operator
\begin{equation*}
\Lap_{\theta}=\Lap_{g_{\theta}}-T^2.
\end{equation*}
If $\left\{T,X_1,\dots,X_n,Y_1,\dots,Y_n\right\}$ is an orthomormal basis for the tangent space, such that $Y_i=JX_i$ for every $i=1,\dots, n$, then the Webster scalar curvature $\S_{\theta}$ is given by
\begin{equation}\label{eq: Webster scalar curvature}
\S_{\theta}=\sum_{j=1}^{n}\left(Ric_{g_{\theta}}(X_j,X_j)+Ric_{g_{\theta}}(Y_j,Y_j)\right)+4n.
\end{equation}
here we have denoted the Ricci tensor by $Ric_{g_{\theta}}$.
Let $(M,\theta, g_{\theta}, \phi)$ be a contact Riemmanian manifold and let $u$ be a positive function on $M$, we consider a new manifold $(M,\tilde\theta, \tilde g_{\theta}, \tilde\phi)$, where $\tilde{\theta}$ is the contact form defined by
\begin{equation*}
\tilde{\theta}=u^{p-2}\theta, \quad \quad p=\frac{2n+2}{n},
\end{equation*}
with $\phi$ and $\tilde{\phi}$ acting in the same way on $\text{ker}(\theta)=\text{ker}(\tilde\theta)$. The scalar curvatures $\S_{\theta}$ and $\S_{\tilde{\theta}}$ are related by the following identity (see \cite{Tan89})
\begin{equation}\label{eq: id1}
-\Lap_{\theta}u+\frac{n}{4(n+1)}\S_{\theta}u=\frac{n}{4(n+1)}\S_{\tilde{\theta}} u^{p-1}.
\end{equation}

\noindent
Now let $\H^n\simeq \R\times \C^n\simeq\R\times\R^{2n}$ be the Heisenberg group. We denote the coordinates by
$$w=(t,z)=(t,x_1,y_1,\dots,x_{2n}, y_{2n})$$
and the group law
\begin{equation*}
w\cdot  w'=(t,z)\cdot(t',z')=(t+t'+2\text{Im}(z\ol{z'}), z+z')\quad\forall\;  w,w'\in\H^n,
\end{equation*}
where $\text{Im}(\cdot)$ denotes the imaginary part of a complex number and $z\ol{z'}$ is the standard Hermitian inner product in $\C^n$. Left translations on $\H^n$ are defined by
\begin{equation*}
\tau:\H^n\to\H^n \qquad \tau_{w}(w')=w\cdot w'\quad \forall\; w\in\H^n
\end{equation*}
and dilations are
\begin{equation*}
\delta_{\lambda}:\H^n\to\H^n \qquad \delta_{\lambda}(t,z)=(\lambda^2  t,\lambda z)\quad \forall\;\lambda>0.
\end{equation*}
We denote by $Q=2n+2$ the homogeneous dimension of $\H^n$ with respect to $\delta_{\lambda}$. On $\H^n$ we consider the  standard contact form
\begin{equation*}
\thetaH{n}=\text{d}t+2\sum_{j=1}^{n}(x_j\d y_j-y_j\d x_j).
\end{equation*}
The canonical orthonormal basis (with respect to $g_{\thetaH{n}}$) of left invariant vector fields on $\H^n$ is
\begin{equation*}
X_j^{\thetaH{n}}=\frac{1}{\sqrt{2}}\left(\frac{\partial}{\partial x_j}+2 y_j\frac{\partial}{\partial t}\right),\quad Y_j^{\thetaH{n}}=-\frac{1}{\sqrt{2}}\left(\frac{\partial}{\partial y_j}-2x_j\frac{\partial}{\partial t}\right),\quad T^{\thetaH{n}}=\frac{\partial}{\partial t}, \quad j=1,\dots,n.
\end{equation*}\\
We set for every $j=1,\dots,n$
\begin{equation}\label{eq: action of J on Hn}
\begin{split}
\phi^{\thetaH{n}} \left(X_{j}^{\thetaH{n}}\right)&=Y_{j}^{\thetaH{n}}\\
\phi^{\thetaH{n}}\left(\frac{\partial}{\partial x_j}\right)&=-\frac{\partial}{\partial y_j}+2x_j\frac{\partial}{\partial t}\\
\phi^{\thetaH{n}}\left(\frac{\partial}{\partial y_j}\right)&=\frac{\partial}{\partial x_j}+2y_j\frac{\partial}{\partial t}\\
\phi^{\thetaH{n}}\left(\frac{\partial}{\partial t}\right)&=0
\end{split}
\end{equation}

\noindent
Now let $\mathbb{S}^{2n+1}\subseteq \C^{n+1}$ be the unit sphere
$$S^{2n+1}=\{\zeta\in\C^{n+1}:\; |\zeta|=1\}.$$
We denote by $\thetaS{n}$ its standard contact form
\begin{equation*}
\thetaS{n}=\sum_{j=1}^{n+1}(u_j\d v_j-v_j\d u_j), \qquad \text{with }\zeta_j=u_j+iv_j
\end{equation*}
and by $g_{\thetaS{n}}$ the related standard metric. Then the Reeb vector field is
\begin{equation*}
T^{\thetaS{n}}=\sum_{j=1}^{n+1} -v_j\frac{\partial}{\partial u_j}+u_j\frac{\partial}{\partial v_j}
\end{equation*}
and the Webster scalar curvature is
\begin{equation*}
\S_{\thetaS{n}}=4n^2+4n.
\end{equation*}

\noindent
The Cayley transform identifies the Heisenberg group with the unit sphere minus a point. More precisely, for $P_S\in\Sph^{2n+1}$, $P_S=(0,\dots,0,-1)$ the Cayley transform is $\Cay:\H^n\to \Sph^{2n+1}\setminus \{P_s\}$
\begin{equation*}
\Cay(t,z)=(\zeta_1,\dots,\zeta_{n+1})=\left(\frac{2 z}{1+|z|^2+it},\frac{1-|z|^2-it}{1+|z|^2+it}\right)
\end{equation*}
or equivalently
\begin{equation*}
\Cay(t,x_1,y_1,\dots,x_n,y_n)=(u_1,v_1\dots,u_{n+1},v_{n+1})
\end{equation*}
with
\begin{equation*}
\begin{split}
u_j&=2\frac{x_j(1+|z|^2)-ty_j}{t^2+(1+|z|^2)^2},\quad v_j=2\frac{tx_j+(1+|z|^2)y_j}{t^2+(1+|z|^2)^2}, \quad j=1,\dots,n \\
u_{n+1}&=\frac{1-|z|^4-t^2}{t^2+(1+|z|^2)^2},\quad v_{n+1}=\frac{2t}{t^2+(1+|z|^2)^2}.
\end{split}
\end{equation*}
Then the contact forms $\thetaH{n}$ and $\thetaS{n}$ are related by the following identity
\begin{equation}\label{eq:relation theta sphere and H}
\Cay^*\thetaS{n}=\frac{2}{t^2+(1+|z|^2)^2}\thetaH{n}.
\end{equation}\\
It the sequel we will need the inverse of $\Cay$, that is $\Cay^{-1}:\;\Sph^{2n+1}\setminus\{P_S\}\to \H^{n}$
\begin{equation*}
\Cay^{-1}(\zeta_1,\dots,\zeta_{n+1})=(t,z_1,\dots,z_n)=\left(\text{Re}\left(i\frac{1-\zeta_{n+1}}{1+\zeta_{n+1}}\right),\frac{\zeta_1}{1+\zeta_{n+1}},\dots,\frac{\zeta_n}{1+\zeta_{n+1}}\right).
\end{equation*}
or equivalently
\begin{equation*}
\Cay^{-1}(u_1,v_1\dots,u_{n+1},v_{n+1})=(t,x_1,y_1,\dots,x_n,y_n)
\end{equation*}
where
\begin{equation*}
t=\frac{2v_{n+1}}{v_{n+1}^2+(1+u_{n+1})^2}, \quad x_j=\frac{u_j(1+u_{n+1})+v_jv_{n+1}}{v_{n+1}^2+(1+u_{n+1})^2}, \quad y_j=\frac{v_j(1+u_{n+1})-u_jv_{n+1}}{v_{n+1}^2+(1+u_{n+1})^2}
\end{equation*}
with $j=1,\dots,n$.

\section{Explicit construction of the singular contact structure}
\noindent
Here we will construct an explicit contact form $\theta_{k,N}$ on $\Sph^{2n+1}\setminus \Sph^{2k+1}$ which will be conformal to the standard CR contact structure $\thetaS{n}$ of $\Sph^{2n+1}$, having complete Webster metric and constant Webster scalar curvature.\\
First of all we transform the problem on $\Sph^{2n+1}$ into a problem on $\H^n$ using the Cayley transform. In $\C^{n+1}$ we choose coordinates so that the equatorial sphere $\Sph^{2k+1}$ is defined by
\begin{equation*}
\Sph^{2k+1}:=\{\zeta\in \C^{n+1}\; : \zeta=(\zeta_1,\dots,\zeta_k,0,\dots,0,\zeta_{n+1}),\quad |\zeta|=1\}\subseteq \Sph^{2n+1},
\end{equation*}
then we stereographically project $\Sph^{2n+1}$ using $\Cay^{-1}$. Notice that, with this choice of coordinates, the sphere $\Sph^{2k+1}$ is projected down into $\H^k$, so now we consider $\H^{n}$ endowed with the standard contact form $\thetaH{n}$ and we split
\begin{equation*}
\H^{n}\simeq \R\times \C^{n}\simeq \R\times \R^{2k}\times \R^{2(n-k)}\simeq \H^k\times \R^{2(n-k)}
\end{equation*}
with coordinates
\begin{equation*}
(t,z_1,\dots,z_{n})\simeq (t,x_1,y_1,\dots,x_{2n},y_{2n})\simeq (t,x_1,y_1,\dots,x_{k},y_{k},\hat z)
\end{equation*}
where $z_j=x_j+iy_j$, $j=1,\dots,n$. Then, let us set $n-k=N+1$,  and $M=\H^k\times\R\times\Sph^{2N+1} \subseteq \H^k\times\R\times\R^{2(N+1)}$ and the map $\varphi:\; \H^n\to M$
\begin{equation}\label{eq: def varphi}
\varphi(t,x_1,y_1,\dots,x_{k},y_{k},\hat z)=\left(t,x_1,y_1,\dots, x_k,y_k,s,\xi_1,\eta_1,\dots,\xi_{N+1},\eta_{N+1}\right)
\end{equation}
which is the identity on $t,\; x_i,\; y_i,$ for $i=1,\dots,k$ and
\begin{equation*}
\begin{split}
s=\ln(|\hat{z}|), \quad \xi_{j}=\frac{x_{k+j}}{|\hat{z}|},\quad \eta_{j}=\frac{ y_{k+j}}{|\hat{z}|}\quad j=1,\dots,N+1.
\end{split}
\end{equation*}\noindent
On $M$ we consider the contact form
\begin{equation*}
 \theta_{k,N}:=\thetaS{N}+\frac{e^{-2s}}{2}\thetaH{k}.
\end{equation*}
The following Proposition shows the relationship between $\theta_{k,N}$, $\thetaH{n}$ and $\thetaS{n}$
\begin{proposition} Using the notation above we have
\begin{equation*}
(\varphi^{-1})^*\thetaH{n}=2e^{2s}\theta_{k,N}
\end{equation*}
and
\begin{equation}\label{eq:relation theta sphere and nk}
(\varphi^{-1}\circ \Cay)^*\thetaS{n}= \frac{4e^{2s}}{t^2 + \left(1+\sum_{i=1}^k(x_i^2 + y_i^2)+ e^{2s}\right)^2} \theta_{k,N}.
\end{equation}
\end{proposition}

\begin{proof}
By straightforward computation we find
\begin{align*}
(\varphi^{-1})^* \d x_{j+k}&=e^s\xi_j\d s+e^s\d \xi_j,\qquad (\varphi^{-1})^* \d y_{j+k}=e^s\eta_j\d s+e^s\d \eta_j,\quad j=1,\dots,N+1
\end{align*}
hence
\begin{equation}\label{eq: varphithetaHn}
(\varphi^{-1})^*\thetaH{n}=\d t+2\sum_{i=1}^k (x_i\d y_i-y_i\d x_i)+2e^{2s}\sum_{j=1}^{N+1}(\xi_j\d \eta_j-\eta_j\d \xi_j)=\thetaH{k}+2e^{2s}\thetaS{N}.
\end{equation}
Then, equality \eqref{eq:relation theta sphere and nk} follows from  \eqref{eq:relation theta sphere and H} and the identity above.
\end{proof}

\begin{remark}
Let us explicitly note that one can see the contact form $\theta_{k,N}$ defined on $\H^{n}\setminus \H^{k}$ with the singularity along $\H^{k}$, just by letting $r=|\hat{z}|$ (see also formula (\ref{eq: singular radial}) in the sequel). We chose the variable $s=\ln(|\hat{z}|)$ in order to make the computations easier.
\end{remark}

\noindent
From now on we will consider the contact manifold $(M, \theta_{k,N})$, where
$$M=\H^k \times \R \times \Sph^{2N+1}$$
with coordinates
$$(t, x_1, \ldots, x_{k}, y_1, \ldots, y_{k}, s,  \xi_{1}, \ldots, \xi_{N+1}, \eta_{1}, \ldots, \eta_{N+1})=
(t, x, y, s,  \xi, \eta) , \quad |(\xi, \eta)|=1$$
and contact form  $\theta_{k,N}$. Moreover we consider the $(1,1)$-tensor
\begin{equation}\label{eq:phi M}
\phi=\phi^{\theta_{k,N}}=\d\varphi\circ\phi^{\thetaH{n}}\circ \d\varphi^{-1}
\end{equation}
and the metric $g=g_{\theta_{k,N}}$ defined by \eqref{eq: metric associated to theta}. We will show that the Webster scalar curvature $\S_{\theta_{k,N}}$ is constant. In order to compute $\S_{\theta_{k,N}}$ we choose a particular orthonormal basis for $T_pM$. Let us notice that, since $\theta_{k,N}= \thetaS{N}+\frac{e^{-2s}}{2}\thetaH{k}$, the Reeb vector field $T^{\theta_{k,N}}$ of $(M,\theta_{k,N})$ is the Reeb vector field of $(\Sph^{2N+1},\thetaS{N})$, so
\begin{equation*}
T:=T_{\theta_{k,N}}=\sum_{j=1}^{N}\left(-\eta_j\frac{\partial}{\partial \xi_j}+\xi_j\frac{\partial}{\partial\eta_j}\right).
\end{equation*}
We consider the following vector fields in $\text{ker}(\theta_{k,N})$
\begin{equation*}
X_0=\frac{\partial}{\partial s},\qquad Y_0=2e^{2s}\frac{\partial}{\partial t}-T, \qquad X_i=\sqrt{2}e^sX_i^{\thetaH{k}}, \qquad Y_i=\sqrt{2}e^sY_i^{\thetaH{k}},\qquad i=1,\dots, k.\\
\end{equation*}
By straightforward computations we have
\begin{align*}
\d\varphi^{-1}(X_0)&=\sum_{j=1}^N\left(x_{j+k}\frac{\partial}{\partial x_{j+k}}+y_{j+k}\frac{\partial}{\partial y_{j+k}}\right)=\sqrt{2} \sum_{j=1}^N\left(x_{j+k} X_{j+k}^{\thetaH{n}}-y_{j+k} Y_{j+k}^{\thetaH{n}}\right)\\
\d\varphi^{-1}(Y_0)&=2|\hat{z}|^2\frac{\partial}{\partial t}+\sum_{j=1}^N\left(y_{j+k}\frac{\partial}{\partial x_{j+k}}-x_{j+k}\frac{\partial}{\partial y_{j+k}}\right)=\sqrt{2} \sum_{j=1}^N\left(y_{j+k} X_{j+k}^{\thetaH{n}}+x_{j+k} Y_{j+k}^{\thetaH{n}}\right)
\end{align*}
then, recalling the identities $J^{\thetaH{n}} X_{j}^{\thetaH{n}}= Y_{j}^{\thetaH{n}}$ for every $j=1,\dots,n$, the above computations show that
\begin{equation}\label{eq: relation1}
J^{\thetaH{n}}\d \varphi^{-1}(X_0)=\d \varphi^{-1}(Y_0).
\end{equation}
Similarly, for $i=1,\dots,k$ we have
\begin{align*}
\d\varphi^{-1}(X_i)=|\hat{z}|\sqrt{2}X_i^{\thetaH{n}}\quad \text{and}\quad \d\varphi^{-1}(Y_i)=|\hat{z}|\sqrt{2}Y_i^{\thetaH{n}}
\end{align*}
so
\begin{equation}\label{eq: relation2}
J^{\thetaH{n}}\d \varphi^{-1}(X_i)=\d \varphi^{-1}(Y_i),\quad i=1,\dots,k.
\end{equation}

\noindent
Now we notice that the metric and the endomorphism $\phi$ induced from $(M,\theta_{k,N},g, \phi$) on $\Sph^{2N+1}\subseteq M$ are the standard ones. Indeed
\begin{align*}
\d\varphi^{-1}(T)&=\d\varphi^{-1}\left(2e^{2s}\frac{\partial}{\partial t}-Y_0\right)=2|\hat z|^2\frac{\partial}{\partial t}-\sqrt{2}\sum_{j=1}^N\left(y_{j+k} X_{j+k}^{\thetaH{n}}+x_{j+k} Y_{j+k}^{\thetaH{n}}\right)\\
W_j&:=\d\varphi^{-1}\left(\frac{\partial}{\partial \xi_j}+\eta_jT\right)=|\hat z|\frac{\partial}{\partial x_{j+k}}+\frac{y_{j+k}}{|\hat z|}\d\varphi^{-1}(T),\quad j=1,\dots,N+1\\
Z_j&:=\d\varphi^{-1}\left(\frac{\partial}{\partial \eta_j}+\xi_jT\right)=|\hat z|\frac{\partial}{\partial y_{j+k}}-\frac{x_{j+k}}{|\hat z|}\d\varphi^{-1}(T)\quad j=1,\dots,N+1.
\end{align*}
Thus, recalling \eqref{eq: action of J on Hn},
\begin{align*}
\phi^{\thetaH{n}}\left(\d \varphi^{-1}(T)\right)&=-\sqrt{2} \sum_{j=1}^N\left(x_{j+k} X_{j+k}^{\thetaH{n}}-y_{j+k} Y_{j+k}^{\thetaH{n}}\right)=\d\varphi^{-1}(X_0)\\
\phi^{\thetaH{n}}\left(W_j\right)&=-Z_j\quad j=1,\dots,N+1,
\end{align*}
which imply respectively
\begin{align*}
\phi^{\theta_{k,N}}\left(T\right)=X_0\\
\phi^{\theta_{k,N}}\left(\frac{\partial}{\partial \xi_j}+\eta_jT\right)&=-\left(\frac{\partial}{\partial \eta_j}+\xi_jT\right).
\end{align*}
Since the metric and the endomorphism $\phi$ induced on $\Sph^{2N+1}$ from $(M,\theta_{k,N},g)$ are the standard ones, locally, at each point $p\in M \;$ we can consider $2N$ orthonormal geodesic Killing vector fields for $(\Sph^{2N+1},\thetaS{N})$
\begin{equation}\label{eq: def UV}
U_j, \qquad V_j, \qquad j=1,\dots,N
\end{equation}
such that $J^{\thetaH{n}}\d \varphi^{-1}(U_j)=\d \varphi^{-1}(V_j)$ and $U_j,V_j\in\text{ker}\left(\thetaS{N}\right).$\\
We define the set $\mathcal{B}:=\{X_0,Y_0,X_1,\dots,X_k,Y_1,\dots,Y_k,T,U_1,\dots, U_N,V_1,\dots,V_N\}$.

\begin{proposition}
The set $\mathcal{B}$ is an orthonormal basis for $T M$, and $J^{\theta_{k,N}}=\phi^{\theta_{k,N}}|_{\ker{\theta_{k,N}}}$ acts as follows
\begin{equation}\label{eq: action of Jnk}
J^{\theta_{k,N}}X_0=Y_0,\quad J^{\theta_{k,N}}X_i=Y_i,\quad J^{\theta_{k,N}}U_j=V_j.
\end{equation}
\end{proposition}
\begin{proof} Identities in \eqref{eq: action of Jnk} follows from \eqref{eq: relation1}, \eqref{eq: relation2} and the definition of $U_j$'s and $V_j$'s. Now it is straightforward to check that $\mathcal{B}$ is orthnonormal using the definition of $g$ (see \eqref{eq: metric associated to theta}):
\begin{equation*}
g(Z,W)=-\frac{1}{2}\d\theta_{k,N}(Z,\phi W),\qquad \d\theta_{k,N}=\d\thetaS{N}-e^{-2s}\d s\wedge\thetaH{k}+\frac{e^{-2s}}{2}\d \thetaH{k}
\end{equation*}
$Z,W\in TM.$ We just compute $g(X_0,X_0)$ as an example:
\begin{align*}
g(X_0,X_0)=-\frac{1}{2}\d \theta_{k,N}(X_0,Y_0)=-\frac{1}{2}(-e^{-2s})2e^{2s}=1.
\end{align*}
\end{proof}

\noindent
We will compute the Webster scalar curvature $\S_{\theta_{k,N}}$ with the aid of three lemmas. Let $\n$ be the Levi-Civita connection on $(M,\theta_{k,N},g)$, then we have the following

\begin{lemma}\label{lem: connection coefficients 1}
For every $j=1,\dots, N$ we have
\begin{align*}
&\n_T T=0 			&\quad &\n_{T}U_j=V_j 		&\quad &\n_{T}V_j=-U_j\\
&\n_{U_j}T=-V_j,	&\quad &\n_{U_j}U_j=0,		&\quad &\n_{U_j}V_j=T,\\
&\n_{V_j}T=U_j, 	&\quad &\n_{V_j}U_j=-T,		&\quad &\n_{V_j}V_j=0.
\end{align*}
\end{lemma}

\begin{proof}
Since $T$, $U_j$ and $V_j$ are geodesic we have $\n_{T} T=0,\;\n_{U_j} U_j=0,\;\n_{V_j} V_j=0$ for every $j=1,\dots, N$. Moreover $U_j$'s are Killing vector fields on $(\Sph^{2N+1},g_{\thetaS{N}})$, so
\begin{equation}\label{eq: Killing}
g(\n_X U_j,Y) +g(X,\n_Y U_j)=0 \text{ for every } X,Y\in T\Sph^{2N+1}.
\end{equation}
\noindent
We denote by $\tilde{J}$ the complex structure on $\C^{N+1}$,  by $\nu$ the outward unit normal to $\Sph^{2N+1}$ and by $\tilde{g},\tilde{\n}$ the standard metric and Levi Civita connection of $\C^{N+1}$. We will us the same notation for the induced metric and connection on $\Sph^{2N+1}$. Then on $T\Sph^{2N+1}\subseteq TM$,  $\tilde{J}T=\nu$ and $\tilde{J}, J^{\theta_{k,N}}$ have the same actions on $\text{ker}{\thetaS{N}}\subseteq\text{ker} \theta_{k,N}$ and $\tilde{g}=g_{\theta_{k,N}}$. Also, we denote by $h(Z,W)=\tilde{g}(\tilde{\n_{Z}}W,-\nu)$, $Z,W\in T\Sph^{2N+1}$, the second fundamental form of $M$ restricted to $\Sph^{2N+1}$. Notice that, with respect to the the basis $\{T,U_1,V_1,\dots,U_N,V_N\}$, the second fundamental form $h$ is the $(2N+1)\times (2N+1)$ identity matrix.
The following compatibility relations hold
\begin{equation}\label{eq: comp rela}
\tilde{g}(\cdot,\cdot)=\tilde{g}(\tilde{J}\cdot,\tilde{J}\cdot), \quad
\tilde{\n}\tilde{J}\cdot=\tilde J\tilde \n \cdot 
\end{equation}
Then for every $j,l=1,\dots, N$ we have
\begin{align*}
g(\n_T U_j,U_l)&\overset{\eqref{eq: Killing}}{=}-g(T,\n_{U_l} U_j)=-\tilde{g}(T,\tilde\n_{U_l} U_j) =\\
=-\tilde{g}(\tilde{J}T,\tilde\n_{U_l} \tilde{J}U_j)&\overset{\eqref{eq: comp rela}}{=}\-\tilde{g}(\nu,\tilde\n_{U_l} V_j)=h(U_l,V_j)=0 \\
\intertext{and similarly}
g(\n_T U_j,V_l)&\overset{\eqref{eq: Killing}}{=}-g(T,\n_{V_l} U_j)\overset{\eqref{eq: comp rela}}{=}h(V_l,V_j)=\delta_{jl}\\
g(\n_T U_j,T)&\overset{\eqref{eq: Killing}}{=}-g(T,\n_{T} U_j)\overset{\eqref{eq: comp rela}}{=}h(T,V_j)=0.
\end{align*}
Also
\begin{equation*}
g(\n_T U_j,X_i)=0,\qquad g(\n_T U_j,Y_i)=0 \quad \text{for every } i=0,\dots,k.
\end{equation*}
Thus
\begin{equation*}
\n_{T}U_j=V_j  \quad\text{for every } j=0,\dots,N.
\end{equation*}
\noindent
Recalling that $V_j's$ are geodesic Killing vector fields, the same argument gives
\begin{equation*}
\n_{T}V_j=-U_j \quad \text{for every } j=0,\dots,N.
\end{equation*}
Moreover
\begin{align*}
g(\n_{ U_j}T,U_l)&\overset{\eqref{eq: comp rela}}{=} h(U_j,V_j)=0\\
g(\n_{ U_j}T,V_l)&\overset{\eqref{eq: comp rela}}{=} -h(U_j,U_l)=-\delta_{jl}\\
g(\n_{ U_j}T,T)&=g(\n_{ U_j}T,X_i)=g(\n_{ U_j}T,Y_i)=0 \qquad\text{ for i=0,\dots, k}.
\end{align*}
hence
$$\n_{U_j}T=-V_i.$$
\noindent
Since $U_j$'s are geodesic we have $\tilde{\n}_{U_j}U_j=-\nu$, from which we get
\begin{equation*}
\tilde{\n}_{U_j}V_j =\n_{ U_j }V_j=T.
\end{equation*}
Analogous computations give $\n_{V_j}T=U_j$ and $\n_{V_j}U_j=-T$.
\end{proof}

\noindent
In the sequel we will use the following formula to compute some covariant derivatives:
\begin{multline}\label{eq: connection coefficients formula}
g(\nabla _{X}Y,Z)={\tfrac {1}{2}}{\Big \{}X{\bigl (}g(Y,Z){\bigr )}+Y{\bigl (}g(Z,X){\bigr )}-Z{\bigl (}g(X,Y){\bigr )}+\\
+g{\bigl (}[X,Y],Z{\bigr )}-g{\bigl (}[Y,Z],X{\bigr )}-g{\bigl (}[X,Z],Y{\bigr )}{\Big \}},
\end{multline}
where $X,Y,Z\in TM$. So, first we compute the necessary commutators.

\begin{lemma} \label{lem: commutators}
For every $i,l=1,\dots,k$ and every $j=1,\dots,N$, we have
\begin{align*}
&[X_0,Y_0]=2Y_0+2T,             &\quad &[X_0,X_i]=X_i,    &\quad &[X_0,Y_i]=Y_i,	&\quad &[X_0,T]=0,\\
&[X_0,U_j]=0,                   &\quad &[X_0,V_j]=0,      &\quad &[Y_0,X_i]=0,		&\quad &[Y_0,Y_i]=0,\\
&[Y_0,T]=0,                     &\quad &[Y_0,U_j]=-2V_j,  &\quad &[Y_0,V_j]=2U_j,	&\quad &[X_i,X_l]=0,\\
&[X_i,Y_l]=\delta_{il}(2Y_0+2T),&\quad &[X_i,T]=0,        &\quad &[X_i,U_j]=0,		&\quad &[X_i,V_j]=0,\\
&[Y_i,Y_l]=0,                   &\quad &[Y_i,T]=0,        &\quad &[Y_i,U_j]=0,		&\quad &[Y_i,V_j]=0,\\
&[U_j,T]=-2V_j,                 &\quad &[V_j,T]=2U_j.\\
\end{align*}
\end{lemma}

\begin{proof}
Using Lemma \ref{lem: connection coefficients 1}, for every $j=1,\dots, N$ we compute
\begin{align*}
[U_j,T]&=\n_{U_j}T-\n_T U_j=-2V_j\\
[V_j,T]&=\n_{V_j}T-\n_T V_j=2U_j
\end{align*}
from which we get
\begin{align*}
[Y_0,U_j]&=[-T,U_j]=-2V_j \\
[Y_0,V_j]&=[-T,V_j]=2U_j.
\end{align*}
All the other commutators are computed using the explicit expression of the vector fields involved and the fact that $M$ is a product manifold.
\end{proof}

\noindent
Using \eqref{eq: connection coefficients formula} and Lemma \ref{lem: commutators} we compute the following covariant derivatives:
\begin{lemma} \label{lem: connection coefficients 2}
For every $i,l=1,\dots,k$ and every $j=1,\dots,N$, we have
\begin{align*}
&\n_{X_0}X_0 =0,				&\quad &\n_{X_0}Y_0 =T,  					&\quad &\n_{X_0}X_i=0, 		&\quad &\n_{X_0}Y_i=0,\\
&\n_{X_0}T=-Y_0, 				&\quad &\n_{X_0}U_j=0, 						&\quad &\n_{X_0}V_j=0, 		&\quad &\n_{Y_0}X_0=-2Y_0-T,\\
&\n_{Y_0} Y_0=2X_0,				&\quad &\n_{Y_0} X_i=-Y_i, 					&\quad &\n_{Y_0} Y_i=X_i, 	&\quad &\n_{Y_0}T=X_0,\\
&\n_{Y_0}U_j =-2V_j, 			&\quad &\n_{Y_0}V_j =2U_j, 					&\quad &\n_{X_i}X_0=-X_i, 	&\quad &\n_{X_i} Y_0=-Y_i, \\
&\n_{X_i} X_l=\delta_{il}X_0, 	&\quad &\n_{X_i} Y_l=\delta_{il}(T+Y_0), 	&\quad &\n_{X_i}T=-Y_i, 	&\quad &\n_{X_i}U_j=0\\
&\n_{X_i}V_j =0, 				&\quad &\n_{Y_i}X_0=-Y_i, 					&\quad &\n_{Y_i} Y_0=X_i, 	&\quad &\n_{Y_i} X_l=-\delta_{il}(T+Y_0),\\
&\n_{Y_i} Y_l=\delta_{il}X_0, 	&\quad &\n_{Y_i}T =X_i, 					&\quad &\n_{Y_i}U_j =0, 	&\quad &\n_{Y_i}V_j=0,\\
&\n_{T}X_0=-Y_0, 				&\quad &\n_{T} Y_0=X_0, 					&\quad &\n_{T} X_i=-Y_i, 	&\quad &\n_{T} Y_i=X_i,\\
&\n_{U_j}X_0=0, 				&\quad &\n_{U_j} Y_0=0, 					&\quad &\n_{U_j} X_i=0, 	&\quad &\n_{U_j} Y_i=0,\\
&\n_{V_j}X_0=0, 				&\quad &\n_{V_j} Y_0=0, 					&\quad &\n_{V_j} X_i=0, 	&\quad &\n_{V_j} Y_i=0.
\end{align*}
\begin{proof}
Since $\mathcal{B}$ is an orthonormal basis, formula \eqref{eq: connection coefficients formula} reduces to
\begin{equation*}
g(\nabla _{X}Y,Z)=\frac {1}{2}\left\{ g{\bigl (}[X,Y],Z{\bigr )}-g{\bigl (}[Y,Z],X{\bigr )}-g{\bigl (}[X,Z],Y{\bigr )}\right\},	\quad \text{for every } X,Y,Z\in\mathcal{B}.
\end{equation*}
Here we compute $\n_{X_0}X_0$ as an example, the other covariant derivatives are computed similarly. Recalling Lemma \ref{lem: commutators}, for every $i=1,\dots,k$ and $j=1,\dots,N$ we have
\begin{align*}
g\left(\n_{X_0}X_0,X_0 \right)&=0,\\
g\left(\n_{X_0}X_0, Y_0\right)&=-g\left([X_0, Y_0],X_0\right)=-g\left(2Y_0+2T,X_0\right)=0,\\
g\left(\n_{X_0}X_0, X_i\right)&=-g\left([X_0, X_i],X_0\right)=0,\\
g\left(\n_{X_0}X_0, Y_i\right)&=-g\left([X_0, Y_i],X_0\right)=0,\\
g\left(\n_{X_0}X_0, T\right)&=-g\left([X_0, T],X_0\right)=0,\\
g\left(\n_{X_0}X_0, U_j\right)&=-g\left([X_0, U_j],X_0\right)=0,\\
g\left(\n_{X_0}X_0, V_j\right)&=-g\left([X_0, V_j],X_0\right)=0.
\end{align*}
Thus $\n_{X_0}X_0=0$.
\end{proof}
\end{lemma}

\noindent
Now we are ready to conclude the proof of Theorem \ref{thm: explicit singular contact form}
\begin{proof}[Proof of Theorem \ref{thm: explicit singular contact form} ]
It remains to compute $\S_{\theta_{k,N}}$. For every $W\in\mathcal{B}$ we have
\begin{align}\label{eq: Ricci}
Ric_{g}(W,W)=\sum_{Z\in\mathcal{B}} g\left(\n_{Z}\n_{W}W-\n_{W}\n_{Z}W-\n_{[Z,W]}W,Z\right).
\end{align}
We explicitly compute  $Ric_{g}(X_i,X_i)$ for every $i=1,\dots,k.$  By Lemma \ref{lem: connection coefficients 1} and Lemma \ref{lem: connection coefficients 2} we have
\begin{align*}
Ric_{g}(X_i,X_i)&=\sum_{Z\in\mathcal{B}} g\left(\n_{Z}\n_{X_i}X_i-\n_{X_i}\n_{Z}X_i-\n_{[Z,X_i]}X_i,Z\right)\\
&=\sum_{Z\in\mathcal{B}} g\left(\n_{Z}X_0-\n_{X_i}\n_{Z}X_i-\n_{[Z,X_i]}X_i,Z\right)\\
&=g\left(-\n_{X_i}X_i,X_0\right)+ g\left(\n_{Y_0}X_0+\n_{X_i}Y_i,Y_0\right)+\sum_{l=1}^k g\left(\n_{X_l}X_0-\delta_{li}\n_{X_i}X_0,X_l\right)+\\
&\qquad + \sum_{l=1}^k g\left(\n_{Y_l}X_0-\delta_{li}\n_{X_i}(T+Y_0)+\n_{\delta_{li}(2Y_0+2T)}X_i,Y_l\right)+\\
&\qquad + g\left(\n_{T}X_0-\n_{X_i}Y_i,T\right)+ \sum_{l=1}^N g\left(\n_{U_l}X_0,U_l\right)+\sum_{l=1}^N g\left(\n_{V_l}X_0,V_l\right)\\
&=g\left(-X_0,X_0\right)+ g\left(-2Y_0-T+T+Y_0,Y_0\right)+ \sum_{l=1}^k g\left(-X_l+\delta_{li}X_i,X_l\right)+\\
&\qquad + \sum_{l=1}^k g\left(-Y_l-6\delta_{li}Y_i,Y_l\right)+ g\left(-Y_0+T+Y_0,T\right)+0+0\\
&=-1-1+(-k+1)+(-k-6)+1+0+0\\
&=-6-2k.
\end{align*}
Similarly
\begin{align*}
&Ric_{g}(Y_i,Y_i)=-1-1+(-k-6)+(-k+1)+1+0+0=-6-2k\\
&Ric_{g}(X_0,X_0)=0-7-k-k+1+0+0=-6-2k\\
&Ric_{g}(Y_0,Y_0)=-7+0-k-k+1+0+0=-6-2k
\end{align*}
here we have considered \eqref{eq: Ricci} with $W\in\mathcal{B}$, $Z$ running in the ordered basis $\mathcal{B}$ and we have written, in the order, each of the terms in the sum in the right hand side of \eqref{eq: Ricci}.
Moreover since $M=\H^{k}\times \R\times \Sph^{2N+1}$ and $\{T,U_1,V_1,\dots,U_N,V_N\}$ is an orthonormal basis for $T\Sph^{2N+1}$ with respect to the metric $g_{\thetaS{N}}$, we have
\begin{align*}
Ric_{g}(U_j,U_j)&= Ric_{g_{\thetaS{N}}}(U_j,U_j)+\sum_{\substack{Z=X_0,Y_0,X_i,Y_i\\ i=1,\dots,k}} g\left(\n_{Z}\n_{U_j}U_j-\n_{U_j}\n_{Z}U_j-\n_{[Z,U_j]}U_j,Z\right)\\
&=Ric_{g_{\thetaS{N}}}(U_j,U_j)=2N
\intertext{and }
Ric_{g}(V_j,V_j)&=2N.
\end{align*}
Hence, recalling \eqref{eq: Webster scalar curvature} and the definition $N=n-k-1$, we have
\begin{align*}
\S_{\theta_{k,N}}&=(2k+2)(-6-2k)+(N+N)2N+4n\\
&=4\Bigl((N-k)(N+k)+2(N-k)-(N+k)\Bigr)\\
&=4(N+k+2)(N-k-1)
\end{align*}
that is
\begin{equation*}
\S_{\theta_{k,N}}=4(n+1)(n-2k-2).
\end{equation*}
In particular, we notice that $\S_{\theta_{k,N}}$ is  positive for $k<\frac{n-2}{2}$.
\end{proof}

\section{Singularity along a circle}

\noindent
Here we will use the explicit contact structure that we found in order to obtain some existence result as applications.\\
We will need the explicit expression of $\Lap_{\theta_{k,N}}$, which is
\begin{equation*}
\Lap_{\theta_{k,N}}=T^2+\Lap_{\thetaS{N}}+2e^{2s}\Lap_{\thetaH{k}}+4e^{4s}\frac{\partial^2}{\partial t^2}-4e^{2s}T\frac{\partial}{\partial t}+\frac{\partial^2}{\partial s^2}-2(k+1)\frac{\partial}{\partial s}.
\end{equation*}
Indeed we have
\begin{align*}
X_0^2&=\frac{\partial^2}{\partial s^2}\\
Y_0^2&=T^2+4e^{4s}\frac{\partial^2}{\partial t^2}-4e^{2s}T\frac{\partial}{\partial t}\\
X_i^2&=2e^{2s}\left(X_i^{\thetaH{k}}\right)^2\\
Y_i^2&=2e^{2s}\left(Y_i^{\thetaH{k}}\right)^2 \quad \text{for } i=i,\dots,k,
\end{align*}
so
$$\sum_{i=1}^k (X_i^2+Y_i^2)=2e^{2s}\Lap_{\thetaH{k}}$$
and  by Lemma \ref{lem: connection coefficients 2}
\begin{equation*}
\n_{X_0}X_0=0,\quad \n_{Y_0}Y_0=2\frac{\partial}{\partial s},\quad \n_{X_i}X_i=\frac{\partial}{\partial s},\quad \n_{Y_i}Y_i=\frac{\partial}{\partial s} \quad \n_{T}T=0
\end{equation*}
for $i=1,\dots,k$. Hence
\begin{align*}
\Lap_{\theta_{k,N}}&=\Lap_{g_{\theta_{k,N}}}-T^2\\
&= X_0^2-\n_{X_0}X_0+Y_0^2-\n_{Y_0}Y_0+\sum_{i=1}^k (X_i^2+Y_i^2)-\sum_{i=1}^k (\n_{X_i}X_i+\n_{Y_i}Y_i)+\\
&\qquad \qquad +\sum_{j=1}^{N+1} (U_j^2+V_j^2)-\sum_{j=1}^{N+1} (\n_{U_j}U_j+\n_{V_j}V_j)-\n_TT\\
&= \frac{\partial^2}{\partial s^2}+T^2+4e^{4s}\frac{\partial^2}{\partial t^2}-4e^{2s}T\frac{\partial}{\partial t}-2\frac{\partial}{\partial s}+2e^{2s}\Lap_{\thetaH{k}}-2k\frac{\partial}{\partial s}+\Lap_{\thetaS{N}}.
\end{align*}

\noindent
Next we will need a kind of expansion of the Webster scalar curvature. So let us consider \eqref{eq: def varphi} with the additional change of variable $r=\sqrt{2}e^s$. We denote it by $\bar \varphi$. In these coordinates the standard contact form of $\H^n$ is
\begin{equation}\label{eq: singular radial}
\bar{\theta}=(\bar \varphi^{-1})^* \thetaH{n}=\thetaH{k}+r^2\thetaS{N}
\end{equation}
and we will use the notation $\bar\phi=\d\bar{\varphi}\circ\phi^{\thetaH{n}}\circ \d\bar{\varphi}^{-1}$. 
We define $(\hat\theta,\hat\phi)$ as
\begin{equation}\label{eq: captheta}
\hat\theta=\bar\theta+O(r^2)\beta,\quad \hat\phi=\bar\phi+O(r)\psi
\end{equation}
with $\beta$ a one form and $\psi$ a $(1,1)$-tensor, both with smooth coefficients. We have the following

\begin{proposition}
Let $(\hat\theta,\hat\phi)$ be as in \eqref{eq: captheta} and consider $\tilde{\theta}=r^{-2}\hat\theta$. Then the Webster scalar curvature of $(M,\tilde{\theta},\hat\phi)$ is
$$\mathcal{S}_{\tilde{\theta}}=\mathcal{S}_{\theta_{k,N}}+O(r)$$
\end{proposition}
\begin{proof}
The idea is to compute the Webster scalar curvature $\mathcal{S}_{\hat{\theta}}$ and to write the operator $\Lap_{\hat{\theta}}$ in order to use \eqref{eq: id1} with $u=r^{-n}$ to obtain $\mathcal{S}_{\tilde{\theta}}$. It is convenient to consider $\mathcal{B}$ an orthonormal basis of $TM$ with respect to $\bar{g}$ (i.e the metric defined by $\bar{\theta}$ and $\bar{\phi}$ as in \eqref{eq: metric associated to theta}):
\begin{equation}
\begin{split}
\bar T&=\frac{\partial}{\partial t},\qquad \bar{X_0}=\frac{\partial}{\partial r},\qquad \bar{Y_0}=r\frac{\partial}{\partial t}-\frac{1}{r}T^{\thetaS{N}}\\
\quad \bar{X_i}&=X_i^{\thetaH{k}},\qquad \bar{Y_i}=Y_i^{\thetaH{k}},\quad i=1,\dots,k\\
\bar{U}_j&=\frac{1}{r}U_j, \qquad \bar{V}_j=\frac{1}{r}V_j, \qquad j=1,\dots,N
\end{split}
\end{equation}
with $U_j,\;V_j$ defined as in \eqref{eq: def UV}.
We denote by $\hat{g}$ the metric defined by $\hat{\theta}$ and $\hat{\phi}$ as in \eqref{eq: metric associated to theta}. By definition we have
\begin{align*}
\hat{g}(V,W)&=-\frac{1}{2}d\hat{\theta}(V,\hat{\phi}W)\\
&=-\frac{1}{2}\left[ d\bar{\theta}+O(r)dr\wedge\beta+O(r^2)\beta\right](V,\;\hat{\phi}W+O(r)\psi W)\\
&=\bar{g}(V,W)+d\bar{\theta}(V,O(r)\psi W)+O(r)\left[dr\wedge\beta+O(r)d\beta\right](V,\;\hat{\phi}W+O(r)\psi W).
\end{align*}
Since $\beta(V)=O(1)$ and $d\beta(V,W)=O\left(\frac{1}{r}\right)$ for any $V,W\in \mathcal{B}$, we get
\begin{equation*}
\hat{g}(V,W)=\bar{g}(V,W)+O(r).
\end{equation*}
From this last relation it is possible to compute
\begin{align*}
Ric_{\hat{g}}(V,W)&=Ric_{\bar{g}}(V,W)+O\left(\frac{1}{r}\right),\quad \text{for any } V,W\in \mathcal{B}\\
Ric_{\hat{g}}(\hat T,\hat T)&=Ric_{\bar{g}}\left(\bar T,\bar T\right)+O\left(\frac{1}{r}\right)
\end{align*}
where $\hat{T}$ and  $\bar{T}$ are the Reeb vector fields associated to $\hat\theta$ and $\bar\theta$ respectively, and the scalar curvature
$$R_{\hat{g}}=R_{\bar{g}}+O\left(\frac{1}{r}\right).$$
Then, the Webster scalar curvature $\mathcal{S}_{\hat{\theta}}$ is (see \cite[equation (8.2)]{Tan89})

\begin{align*}
\mathcal{S}_{\hat{\theta}}&=R_{\hat{g}}-Ric_{\hat{g}}(\hat T,\hat T)+4n\\
&=R_{\bar{g}}-Ric_{\bar{g}}\left(\bar T,\bar T\right) +4n + O\left(\frac{1}{r}\right)\\
&= O\left(\frac{1}{r}\right).
\end{align*}
The last equality follows from the fact that the standard Webster scalar curvature of $\H^n$ is $\mathcal{S}_{\thetaH{n}}=\mathcal{S}_{\bar{\theta}}=0$. Now we want to write $\Lap_{\hat{\theta}}u$ where $u=u(r)$ is a function that depends only on the coordinate $r$. We recall that
$$\Lap_{\hat{\theta}}u=\Lap_{\hat{g}}u-\hat{T}^2u$$
where $\Lap_{\hat{g}}$ is the metric Laplacian. In particular for $u=u(r)$ we have
\begin{align*}
\Lap_{\hat{g}}u&=\Lap_{\bar g}u+O(1)\frac{\partial u}{\partial r}+O(r)\frac{\partial^2 u}{\partial^2 r}\\
&=\frac{\partial^2 u}{\partial^2 r}+\frac{2n-2k-1}{r}\frac{\partial u}{\partial r}+O(r)\frac{\partial^2 u}{\partial^2 r}
\end{align*}
and since $\hat{T}=\frac{\partial}{\partial t}+O(r)X$ for a suitable vector field $X$, it holds
$$\hat{T}^2u=O(r^2)\frac{\partial^2u}{\partial^2 r}+O(r)\frac{\partial u}{\partial r}.$$
Hence for $u=r^{-n}$ we find
\begin{align*}
\Lap_{\hat{\theta}}(r^{-n})&=n(n+1)r^{-(n+2)}-n(2n-2k-1)r^{-(n+2)}+O\left(r^{-(n+1)}\right)\\
&=-nr^{-(n+2)}\left(n-2k-2+O(r)\right).
\end{align*}
Finally by using \eqref{eq: id1} with $u=r^{-n}$ we get
\begin{align*}
\S_{\tilde{\theta}}&=\frac{4(n+1)}{n} u^{-\frac{n+2}{n}}\left(-\Lap_{\hat{\theta}}u+\frac{n}{4(n+1)}\S_{\hat{\theta}}u\right)\\
&=\frac{4(n+1)}{n} r^{n+2}\left(-\Lap_{\hat{\theta}}(r^{-n})+O\left(r^{-(n+1)}\right)\right)\\
&=4(n+1)(n-2k-2)+O(r).
\end{align*}
\end{proof}

\subsection{Existence by perturbation}

\noindent
In this section, we will follow closely the perturbation approach developed in \cite{MS}. First let us set $L_{\theta}=\Delta_{\theta}-\frac{n}{4(n+1)}\S_{\theta}$. We consider a smooth embedding $\tau:\mathbb{S}^{1}\to \mathbb{S}^{2n+1}$ close to the identity and we want to find contact structures on $\mathbb{S}^{2n+1}\setminus \tau(\mathbb{S}^{1})$ having constant Webster curvature. Namely, we want to solve on $\mathbb{S}^{2n+1}\setminus \tau(\mathbb{S}^{1})$, the problem
$$L_{\thetaS{n}}v+\frac{n}{4(n+1)}\S_{\theta_{k,N}}v^{p-1}=0.$$
This is equivalent to solve the problem
$$L_{\theta(\tau)}v+\frac{n}{4(n+1)}\S_{\theta_{k,N}}v^{p-1}=0,$$
where $\theta(\tau)=u^{\frac{2}{n}}\tau^{*}\thetaS{n}$ and $u$ is the function giving the conformal change from $\thetaS{n}$ to $\theta_{k,N}$. Since we plan to perturb the equation with respect to the diffeomorphism $\tau$ and around the constant solution $1$, we can write the functional
$$K(\tau,w)=L_{\theta(\tau)}(1+w)+\frac{n}{4(n+1)}\S_{\theta_{k,N}}(1+w)^{p-1}.$$
We want then to solve $K(\tau,w)=0$ via the implicit function theorem, after perturbation around $(id,0)$. So we start by linearizing with respect to $w$:
$$\partial_{w}K(\tau,w)_{|(id,0)}=\Delta_{\theta_{k,N}}+2(n-2k-2).$$
We will consider the operator $\Lap_{\theta_{0,n-1}}$ acting on functions invariant under $T$. Then the operator $L$ takes form
$$L=\Delta_{\mathbb{S}^{2n-1}}+4e^{4s}\partial_{t}^{2}+\partial_{s}^{2}-2\partial_{s}.$$
If one now uses the change of variable $r=e^{2s}$, one gets
$$L=\Delta_{\mathbb{S}^{2n-1}}+4r^{2}\partial_{t}^{2}+4r^{2}\partial_{r}^{2}=\Delta_{\mathbb{S}^{2n-1}}+4\Delta_{\mathcal{H}^{2}}$$
where $\mathcal{H}^{2}=H\R^2$ is the standard hyperbolic space of dimension 2.
In the case $k=0$ the linearized equation becomes then,
$$L_{1}=\Delta_{\mathbb{S}^{2n-1}}+4\Delta_{\mathcal{H}^{2}}+2(n-2)$$
So we first investigate its kernel. For this purpose, we move to the unit disk model of the hyperbolic space with coordinates $x=(\sigma,\vartheta,y)$ where $\sigma\in [0,1]$, $\vartheta\in \mathbb{S}^{1}$ and $y\in  \mathbb{S}^{2n-1}$. We introduce then the family of spaces $C^{\nu,\alpha,k}( \mathbb{S}^{2n-1}\times \mathcal{H}^{2})$ that are adapted to the study of singular problems (see \cite{M,MS,MP,MP2}), by
$$C^{k,\alpha,\nu}( \mathbb{S}^{2n-1}\times \mathcal{H}^{2}):=\{u\in C^{k,\alpha}_{loc}( \mathbb{S}^{2n-1}\times \mathcal{H}^{2});\|u\|_{k,\alpha,\nu}<\infty\}$$
where $$\|u\|_{k,\alpha,\nu}=\sup_{x_{1},x_{2}\in  \mathbb{S}^{2n-1}\times \mathcal{H}^{2}}(\sigma_{1}+\sigma_{2})^{-\nu}\Big(\sum_{j=1}^{k}(\sigma_{1}+\sigma_{2})^{j}|\nabla^{j}u|+(\sigma_{1}+\sigma_{2})^{k+\alpha}[\nabla^{k}]_{\alpha}\Big).$$
In these coordinates, we can express the operator $L_{1}$ as follows:
$$
L_{1}=\Big[(1-\sigma^{2})^{2}\partial_{\sigma}^{2}+\frac{(1-\sigma^{2})^{2}}{\sigma}\partial_{\sigma}+\frac{(1-\sigma^{2})^{2}}{\sigma^{2}}\Delta_{\mathbb{S}^{1}}\Big]+\Delta_{ \mathbb{S}^{2n-1}}+2(n-2),
$$
where $\sigma\in (0,1)$. We look for solutions of the form $u=\sum_{i,j}a_{i,j}(\sigma)\phi_{i}\psi_{j}$ where the $\psi_{j}$ are $T$-invariant eigenfunctions of $\Delta_{ \mathbb{S}^{2n-1}}$ with eigenvalue $\lambda_{j}$ and the $\phi_{i}$ are the eigenfunctions of $\Delta_{ \mathbb{S}^{1}}$ with eigenvalue $\mu_{i}$ (see \cite{MS}, formula (2.13) with the squared eigenvalues). This yields the family of equations
$$A_{i,j}a_{i,j}=0$$
where $$A_{i,j}=(1-\sigma^{2})^{2}\Big [ \partial_{\sigma}^{2}+\frac{1}{\sigma}\partial_{\sigma}-\frac{\mu_{i}}{\sigma^{2}}\Big]-\lambda_{j}+2(n-2)$$
This is a Bessel type equation and the singularity at zero and $1$ is regular. Since we are looking for bounded solutions, there is only a unique regular solution to this equation corresponding to the indicial root $\gamma= i\in \mathbb{N}$, that is a function rotationally invariant. So, we move now to the singularity at 1. We set $\rho=1-\sigma^{2}$, then the operator $A_{i,j}$ becomes
$$A_{ij}=4\rho^{2}\Big[(1-\rho)\partial_{\rho}^{2}-\partial_{\rho}\Big]-\frac{\rho^{2}}{1-\rho}\mu_{i}-\lambda_{j}+2(n-2)$$
In this case, the indicial roots take the form
$$\gamma^{\pm}_{j}=\frac{1}{2}\pm \frac{1}{2}\sqrt{1+\lambda_{j}-2(n-2)}.$$
Notice that $\gamma^{-}$ is positive if and only if $\lambda_{j}=0$. Hence, we set $\nu_{0}=\frac{1}{2}$ and the function space that we will take is $C^{2,\alpha,\nu}( \mathbb{S}^{2n-1}\times \mathcal{H}^{2})$ where $\nu<\frac{1}{2}$. The kernel is then
$$\mathcal{K}(\alpha,\nu)=\{u\in C^{2,\alpha,\nu}; Lu=0\}$$
We recall now a result of Mazzeo-Smale \cite[Theorem 4.54]{MS}
\begin{lemma}[\cite{MS}]\label{lemsm}
For $\nu<\frac{1}{2}$, the operator $L_{1}:C^{2,\alpha,\nu}\to C^{0,\alpha,\nu}$ is onto.
\end{lemma}

\noindent
We define the set $\mathcal{T}$ of smooth (let us say $C^{3,\alpha}$ at least) diffeomorphisms $\tau$ such that they preserve the contact structure at $ \mathbb{S}^{1}$, namely $\tau^{*}\theta_{| \mathbb{S}^{1}}=\theta$.
\begin{proposition}
The map $K$ is $C^{\infty}$ from a neighborhood $\mathcal{N}$ of $(id,0)\in \mathcal{T}\times C^{2,\alpha,\nu}( \mathbb{S}^{2n-1}\times \mathcal{H}^{2})$ to $C^{0,\alpha,\nu}( \mathbb{S}^{2n-1}\times \mathcal{H}^{2})$.
\end{proposition}
\begin{proof}
It is clear that $\mathcal{N}$ is mapped to $C^{0,\alpha}_{loc}$. Without loss of generality we can assume that $\theta(\tau)=\theta_{k,N}+O(r^{2})\beta$ so by Proposition 3.3, we compute
$$K(\tau,w)-K(Id,0)=$$
$$=\Delta_{\theta(\tau)}(1+w)-\Delta_{\theta_{k,N}}1-\frac{n}{4(n+1)}(\S_{\theta(\tau)}(1+w)- \S_{\theta_{k,N}})+\frac{n}{4(n+1)}\S_{\theta_{k,N}}((1+w)^{p-1}-1)$$
Clearly $\Delta_{\theta(\tau)}(1+w)-\Delta_{\theta_{k,N}}1\in C^{0,\alpha,\nu}$. Next, we have that $\S_{\theta(\tau)}=\S_{\theta_{k,N}}+O(r)$ hence, the second term also belongs to $C^{0,\alpha,\nu}$ and similarly for the third term. The higher order derivatives of $K$ can be treated in a similar way.
\end{proof}

\begin{theorem}
Let $0<\nu<\frac{1}{2}$, then there exist a closed subspace $W$ such that $C^{2,\alpha,\nu}=W\oplus \mathcal{K}(\alpha,\nu)$ and a smooth map $\Phi:\mathcal{N}\subset \mathcal{T}\times \mathcal{K}(\alpha,\nu)\to W$ such that $K(\tau,w)=0$, where $w=(\Phi(\tau,w_{1}),w_{1})\in W\oplus \mathcal{K}(\alpha,\nu)$.
\end{theorem}
\begin{proof}
The proof is a direct corollary from the implicit function theorem and Lemma \ref{lemsm}.
\end{proof}

\noindent
As a corollary, we get our first application Theorem \ref{thm: perturbation}.

\subsection{Existence by bifurcation}

\noindent
In this last section we will show the existence of another kind of solutions, via bifurcation, following the work \cite{BPS}. We recall again that $L=\Delta_{\mathbb{S}^{2n-1}}+4\Delta_{\mathcal{H}^{2}}$ and we propose to solve the problem
\begin{equation}\label{eqhyp}
-Lu+\frac{n}{4(n+1)}\S_{\theta_{0,n-1}}u=\frac{n}{4(n+1)}\kappa u^{p-1},
\end{equation}
where $\kappa$ is a positive constant. After taking the quotient of $\mathcal{H}^{2}$ by a Fuchsian group $\Gamma \subset PSL(2,\R)$ we can reduce the study to the manifold $M=\C P^{n-1}\times \Sigma_{\Gamma}$, where $\Sigma_{\Gamma}=\mathcal{H}^{2}/\Gamma$ and $\C P^{n-1}=\mathbb{S}^{2n-1}/\mathbb{S}^{1}$ since the vector field $T$ generate an $\mathbb{S}^{1}$ isometric action corresponding to the Hopf fibration. From now on, we will write $\Sigma$ instead of $\Sigma_{\Gamma}$ and we define the space $\mathcal{M}(\Sigma)$ of hyperbolic metrics on $\Sigma$. In this way we can track the change of the hyperbolic structure by using the metrics $g$. Now, given $g\in \mathcal{M}(\Sigma)$, we define the Banach manifold
$$\mathcal{M}_{\Sigma,g}=\left\{u\in H^{1}(M);\int_{M}u^{p}dv_{g}=Vol_{g}(M); u\geq 0\right\},$$
and the functional defined on it
$$\mathcal{A}_{g}(u)=\frac{1}{2}\int_{M}|\nabla_{M,g} u|^{2}+\frac{n}{4(n+1)}\S_{\theta_{0,n-1}}u^{2}dv,$$
where $\nabla_{M,g}=\nabla_{\C P^{n-1}}\oplus 2\nabla_{\Sigma,g}$. Clearly, critical points of $\mathcal{A}_{g}$ lift to  solutions to the problem (\ref{eqhyp}). We notice also that $1$ is always a solution to our problem with $\kappa=\S_{\theta_{0,n-1}}$. We have then,
$$\nabla \mathcal{A}_{g}(u)=L_{M}u+\frac{n}{4(n+1)}\S_{\theta_{0,n-1}}u-\frac{n}{4(n+1)}\kappa u^{p-1},$$
where $L_{M}=-\Delta_{\C P^{n-1}}-4\Delta_{\Sigma,g}$ and
$$J_{\Sigma,g}=\nabla^{2}\mathcal{A}_{g}(1)=L_{M}-2(n-2).$$
We want to investigate the negative eigenvalues of $J_{\Sigma,g}$, which correspond to the Morse index of $\mathcal{A}_{g}$ at the critical point $1$. So we consider the number
$$n_t(\Sigma,g):=\max \{k\in \N  \; : \; \lambda_k(\Sigma,g)\leq t\}$$
where $\lambda_k(\Sigma,g)$ the are the eigenvalues of the Laplacian on $(\Sigma,g)$. The next two lemmas are in \cite{BPS}.

\begin{lemma}[\cite{BPS}]
Let $t>\frac{1}{4}$, and fix $g_{0}\in \mathcal{M}(\Sigma)$, then for any $k\in \mathbb{N}$, there exists $g_{1}\in \mathcal{M}$ such that $n_t(\Sigma,g_{1}) \geq k+ n_t(\Sigma,g_{0})$.
\end{lemma}

\begin{lemma}[\cite{BPS}]\label{lemnd}
Given a hyperbolic surface $\Sigma$, then the set $\mathcal{M}_{\lambda}(\Sigma)=\{g\in \mathcal{M}(\Sigma); \lambda \not \in \sigma(-\Delta_{\Sigma,g})\}$ is open and dense in $\mathcal{M}(\Sigma)$.
\end{lemma}

\noindent
Now we notice that the eigenvalues of $J_{\Sigma,g}$ take the form
$$\lambda_{\ell}=4\lambda_j(\Sigma,g)+\lambda_k(\C P^{n-1})-2(n-2).$$
\begin{corollary}
Let $n\geq 3$, and let $d\in \N$. Then there exists $g\in \mathcal{M}(\Sigma)$ such that $J_{\Sigma,g}$ has at least $d$ negative eigenvalues.
\end{corollary}
\begin{proof}
Indeed, we always have
$$1<2(n-2)<\lambda_1(\C P^{n-1})=4n.$$
Hence, one looks for eigenvalues of the form $\lambda_{\ell}=4\lambda_j(\Sigma,g)-2(n-2)$. Since $2(n-2)>1$, we can always find $g\in \mathcal{M}(\Sigma)$ such that $\sigma(-\Delta_{\Sigma,g})\cap (\frac{1}{4},\frac{1}{4}+\varepsilon)$ is arbitrarily large. Which proves the claim.
\end{proof}

\noindent
In order to prove existence and multiplicity results for our problem, we will show the existence of bifurcation points while perturbing the metric. We will use the following definition of bifurcation \cite{LPZ}:

\begin{definition}
Given two Banach spaces $B_{0}$ and $B_{1}$ and a $C^{1}$-family of submanifolds $[0,1]:\lambda \mapsto D_{\lambda}\subset B_{1}$ and subspaces $[0,1]:\lambda \mapsto E_{\lambda}\subset B_{0}$. We define the fiber bundle $\mathcal{D}=\{(x,\lambda)\in B_{1}\times [0,1]; x\in D_{\lambda}\}$ and similarly for the fiber $\mathcal{E}=\{(y,\lambda)\in B_{0}\times [0,1];y\in E_{\lambda}\}$. Let $F:\mathcal{D}\to \mathcal{E}$ be a $C^{1}$ bundle morphism. Let $\lambda\mapsto x_{\lambda}$ and $\lambda\mapsto y_{\lambda}$ be $C^{1}$ sections of $\mathcal{D}$ and $\mathcal{E}$ respectively. We say that $\lambda_{*}\in [0,1]$ is a bifurcation point of the equation
$$F(x_{\lambda},\lambda)=(y_{\lambda},\lambda)$$
if there exist a sequence $(\lambda_{n})_{n\geq 1}$ and a sequence $x_{n}\in D_{\lambda_{n}}$ such that
\begin{itemize}
\item[i)]$\lim_{n\to \infty}\lambda_{n}=\lambda_{*}$
\item[ii)]$x_{n}\not=x_{\lambda_{n}}$
\item[iii)]$\lim_{n\to \infty}x_{n}=x_{\lambda_{*}}$
\item[iv)]$F(x_{n},\lambda_{n})=(y_{\lambda_{n}},\lambda_{n})$.
\end{itemize}
\end{definition}

\noindent
Now given a path of metrics $[0,1]:t\to g_{t}\in \mathcal{M}$, the manifold $\mathcal{M}_{\Sigma,g_{t}}$, will play the role of $D_{t}$ and $F(u,t)=\nabla \mathcal{A}_{g_{t}}(u)$, in the definition above. We can see the constant solution $1$ as a section of $\mathcal{D}$, that is, $[0,1]:t\mapsto 1_{t}$, and we have
$$F(1,t)=(0,t).$$
We want to show that we have a bifurcation point for $F$ which corresponds to a sequence of solutions to equation (\ref{eqhyp}) that are arbitrarily close to $1$.
\begin{theorem}
Assume that $n\geq 3$. Given $g_{0} \in \mathcal{M}(\Sigma)$, then there exists $g_{0}' \in \mathcal{M}(\Sigma)$ arbitrarily close to $g_{0}$ and a path $(g'_{t})_{t\in[0,1]}$ joining $g_{0}'$ and $g_{1}'$ such that $F$ has at least one bifurcation point $t_{*}\in (0,1)$.
\end{theorem}
\begin{proof}
We use the bifurcation theorem proved in \cite[Theorem A.2]{LPZ}. First, we notice that for all metrics $g\in \mathcal{M}(\Sigma)$ the operator $J_{\Sigma,g}$ is symmetric and Fredholm of index $0$. We consider now a metric $g_{0}\in \mathcal{M}(\Sigma)$. If $J_{\Sigma,g_{0}}$ is degenerate ($\ker J_{\Sigma,g_{0}}\not=0$, so $1$ is a degenerate critical point for $\mathcal{A}_{g_{0}}$), then by Lemma \ref{lemnd}, we can choose $g_{0}'\in \mathcal{M}(\Sigma)$ arbitrarily close to $g_{0}$ and such that $J_{\Sigma,g_{0}'}$ is invertible (i.e. $\mathcal{A}_{g_{0}'}$ is Morse at $1$), so we let $\mu(g_{0})$ its Morse index.  Using Lemma \ref{lemnd}, we can choose yet another metric $g_{1}'\in \mathcal{M}(\Sigma)$ such that $\mathcal{A}_{g_{1}'}$ is Morse at the critical point $1$ and $\mu(g_{1}')-\mu(g_{0}')\not=0$. In order to conclude now, we consider a smooth path $g_{t}'$ connecting $g_{0}'$ to $g_{1}'$ (such a path exists since $\mathcal{M}(\Sigma)$ is path connected). It is enough to notice now that $d_{1}F(\cdot,t)=J_{\Sigma,g_{t}'}$. Hence, the assumptions of the bifurcation theorem \cite{LPZ} are satisfied and we have at least one bifurcation point $t_{*}\in (0,1)$.
\end{proof}

\noindent
As a corollary, we get our second application Theorem \ref{thm: bifurcation}.

\end{document}